\documentclass[11pt,a4paper]{amsart}
\usepackage[T1]{fontenc}
\usepackage{lmodern}
\usepackage[margin=28mm]{geometry}
\usepackage{amsmath,amssymb,mathtools,mathrsfs,booktabs,array,microtype}
\usepackage{cite}
\usepackage[hidelinks,unicode]{hyperref}
\usepackage[nameinlink,noabbrev]{cleveref}
\crefname{lemma}{lemma}{lemmas}
\Crefname{lemma}{Lemma}{Lemmas}
\crefname{proposition}{proposition}{propositions}
\Crefname{proposition}{Proposition}{Propositions}
\crefname{corollary}{corollary}{corollaries}
\Crefname{corollary}{Corollary}{Corollaries}
\crefname{theorem}{theorem}{theorems}
\Crefname{theorem}{Theorem}{Theorems}
\crefname{remark}{remark}{remarks}
\Crefname{remark}{Remark}{Remarks}
\hypersetup{pdftitle={Completed cycle transfer and the conifold gap for (P2,E)},pdfauthor={Xiaobin Li},pdfsubject={Completed cycle regularity and the Nekrasov-Shatashvili conifold boundary condition},pdfkeywords={conifold gap, Gromov-Witten, completed cycles, NS limit}}
\numberwithin{equation}{section}
\newtheorem{theorem}{Theorem}[section]
\newtheorem{proposition}[theorem]{Proposition}
\newtheorem{lemma}[theorem]{Lemma}
\newtheorem{corollary}[theorem]{Corollary}
\theoremstyle{definition}

\theoremstyle{remark}
\newtheorem{remark}[theorem]{Remark}
\newcommand{\PP}{\mathbb P}
\newcommand{\CC}{\mathbb C}
\newcommand{\QQ}{\mathbb Q}

\newcommand{\h}{\hbar}
\newcommand{\dd}{\partial_t}

\newcommand{\FP}{\operatorname{FP}_{t=0}}
\newcommand{\Pol}{\operatorname{PP}_{t=0}}
\newcommand{\Aut}{\operatorname{Aut}}

\newcommand{\cC}{\mathcal C}
\newcommand{\cB}{\mathcal B}
\newcommand{\cN}{\mathcal N}
\newcommand{\cL}{\mathscr L}
\newcommand{\bR}{\overline R}
\newcommand{\bU}{\overline U}
\newcommand{\mA}{\mathsf A}
\newcommand{\mB}{\mathsf B}
\newcommand{\mC}{\mathsf C}
\newcommand{\bracket}[1]{\left\langle #1\right\rangle_{\rho}}
\DeclareMathOperator{\Dur}{Dur}

\allowdisplaybreaks[1]
\title[Completed cycle transfer and the NS conifold gap]{Completed cycle transfer and the conifold gap\protect\\ for $(\mathbb P^2,E)$}
\author{Xiaobin Li}

\thanks{School of Mathematics, Southwest Jiaotong University, Chengdu, China. Corresponding author: Xiaobin Li; e-mail: \texttt{lixiaobin@home.swjtu.edu.cn}.}
\date{\today}
\subjclass[2020]{Primary 14N35; Secondary 05A17, 11F11, 81T30.}
\keywords{Conifold gap, relative Gromov-Witten invariants, completed cycles, Nekrasov-Shatashvili limit, partitions, quasimodular forms}
\begin{document}
\begin{abstract}
Conifold boundary conditions constrain higher genus topological string
free energies, but the cancellation of subleading poles is not implied by
their leading singular behavior. We prove a regularity theorem for connected
elliptic completed cycle traces and use it to establish the
Nekrasov-Shatashvili conifold gap for the maximal contact relative
Gromov-Witten theory of the projective plane with a smooth cubic boundary
and a top Hodge class insertion. The higher genus local/relative
correspondence expresses the ordinary potential in terms of the relative
potential and differentiated lower genus contributions. We identify its
conifold polarized correction with an elliptic trace. A central difference
identity resums each non-vacuum partition weight into a content product
multiplied by a regular exponential. For arbitrary regular parts of the
input potentials and any nome with a simple zero, the normalized
non-vacuum connected contribution is divisible by the conifold coordinate
and by the square of the coupling. The entire polar part and finite part
are consequently determined by the vacuum constants. A triangular
induction transfers the known ordinary local projective plane gap to the
relative theory. We also obtain finite part identities, finite formulas
for regular Taylor coefficients, and reconstruction of the holomorphic
ambiguities. The argument is formal in the coupling and gives holomorphic
remainders at every fixed genus.
\end{abstract}
\maketitle

\section{Introduction and main results}\label{sec:intro}

The conifold gap is a local boundary condition on higher genus topological
string free energies. In a specified vanishing period coordinate and
polarization, it permits one universal singular monomial and excludes all
other negative powers. Together with the holomorphic anomaly equations of
Bershadsky-Cecotti-Ooguri-Vafa \cite{BCOV}, polynomiality
\cite{YY} and modularity \cite{ABK,ASYYZ,HKQ}, it is a central input to
direct integration. The light-state interpretation of conifold singularities
was developed by Strominger \cite{Strominger}, while Ghoshal-Vafa
\cite{GV} identified the universal conifold sector with the self-dual
$c=1$ theory. These physical descriptions motivate the singular model;
they are distinct from a proof of the full gap for a given enumerative
theory. In the refined setting, the two-parameter
partition function and its Nekrasov-Shatashvili limit have their gauge theoretic
counterparts in the $\Omega$-background \cite{NO,NS,IKV,HK,KW}.
These connections make the \emph{complete} polar statement, rather than only
its leading coefficient, a natural problem in mathematical physics.

An explicit gap condition appears in the 2006 preprints of Huang-Klemm
\cite{HKGauge} and Huang-Klemm-Quackenbush \cite{HKQ}, published in 2007
and 2009, respectively. Its role in fixing holomorphic ambiguities was
subsequently developed in the direct-integration work of
Grimm-Klemm-Mari\~no-Weiss \cite{GKMW} and in the local Calabi-Yau
analysis of Haghighat-Klemm-Rauch \cite{HKR}. In these uses of the
holomorphic anomaly equations, imposing a gap as boundary data must be
distinguished from deriving it independently.

Let $E\subset\PP^2$ be a smooth cubic. Bousseau-Fan-Guo-Wu
\cite{BFGW} constructed a higher genus local/relative correspondence,
proved finite generation and a holomorphic anomaly equation for the
maximal contact relative theory of $(\PP^2,E)$ with a top Hodge class
insertion, and formulated its conifold gap. In their normalization this
theory realizes the relevant Nekrasov-Shatashvili sector of local
$\PP^2$. We denote its conifold potential by $R_g$ and the ordinary local
$K_{\PP^2}$ potential by $U_g$. We prove the relative gap by identifying
and controlling the elliptic completed cycle trace in their geometric
correspondence; neither potential is defined by imposing a boundary
condition.

The obstacle is that this correspondence contains products of high
derivatives of lower genus potentials. Individual terms may have
subleading poles, and even the regular parts of lower genus potentials
can affect those poles before the terms are summed. We keep every such
regular term. After separating the constants of the shifted power sums,
a central difference identity converts the remaining exponential weights
into content products. Every nonempty partition contains a zero content
box, and the normalized connected remainder therefore vanishes at the
conifold. This controls both the full principal part and the finite part.
Here "vacuum" refers precisely to the empty partition constants of the
elliptic trace, not to an additional physical sector imposed on the
relative theory.

\subsection{The relative gap}
Put
\begin{equation}\label{eq:coord-intro}
 \xi=1+27q,\qquad
 \phi(\xi)={}_2F_1\left(\frac13,\frac23;1;\xi\right),\qquad
 t=\frac1{\sqrt3}\int_0^\xi\frac{\phi(s)}{1-s}\,ds.
\end{equation}
The branch is fixed by $\phi(0)=1$. The potentials below use the normalized
conifold polarization described in \cref{sec:geometry}, not the naive
substitution of a large radius power series into $q=(\xi-1)/27$.
For $g\ge2$, define
\begin{equation}\label{eq:coefficients}
 u_g=\frac{B_{2g}}{2g(2g-2)},\qquad
 r_g=-\left(1-2^{1-2g}\right)
       \frac{|B_{2g}|}{2g(2g-1)(2g-2)}.
\end{equation}
Here $B_{2g}$ are Bernoulli numbers, with $B_2=1/6$.

\begin{theorem}[Relative conifold gap]\label{thm:main}
For the maximal contact relative theory of $(\PP^2,E)$, in the coordinate
\eqref{eq:coord-intro} and the conifold polarization of
\cref{sec:geometry}, one has, for every $g\ge2$,
\begin{equation}\label{eq:main-gap}
 R_g(t)=r_g t^{2-2g}+H_g(t),\qquad H_g\in\CC\{t\}.
\end{equation}
In particular, the coefficients of $t^{-j}$ vanish for
$1\le j\le2g-3$.
\end{theorem}

Theorem~\ref{thm:main} proves \cite[Conjecture~1.9]{BFGW}. The result
concerns the actual relative invariants, in the normalization specified
below; the relative potential is not defined by imposing a gap.

For the ordinary local $\PP^2$ theory, Coates-Iritani established its
modular and Fock-sheaf description and verified the conifold gap by explicit
computer-algebra calculations for $2\leq g\leq7$
\cite[Remark~1.1.3 and Section~10.8, equation~(127)]{CI}.
This is finite-genus verification, not an all-genus gap theorem.

The ordinary input used here is the all-genus theorem in Andrea Brini's
2025 preprint \emph{Conifold gap and all-genus mirror symmetry for local
$\PP^2$}
(\href{https://arxiv.org/abs/2509.19298v2}{arXiv:2509.19298v2},
Corollary~4.4). After the exact period normalization proved in
\cref{lem:period-scale}, it states
\begin{equation}\label{eq:ordinary-input}
 U_g(t)=u_g t^{2-2g}+O(1),\qquad g\ge2.
\end{equation}
We use this theorem as an external input. The task here is to transfer
its full polar statement through the local/relative correspondence,
not to reprove the ordinary gap. Brini's statistical-mechanical construction
is related to the operator-based matrix models of Mari\~no-Zakany
\cite{MZ}, which reproduce conifold expansions in examples; his
Section~1.2.2 explains this relationship. The matrix-model comparison
provides context for the ordinary input and is not an additional assumption
in the completed cycle transfer below.

\subsection{The formal regularity mechanism}
Set
\begin{equation}\label{eq:universal-model}
 f_0(t)=\frac12t^2\log t-\frac34t^2,\qquad
 \cB(z)=\frac{z}{2\sinh(z/2)}=\sum_{j\ge0}b_jz^{2j},\qquad
 \Phi=\cB(\h\dd)f_0.
\end{equation}
Thus $b_j=B_{2j}(1/2)/(2j)!$. For a partition $\lambda$, write
$c(i,j)=j-i$ for the content of a box and set
\begin{equation}\label{eq:shifted-intro}
 p_k(\lambda)=\sum_{i\ge1}
 \left[\left(\lambda_i-i+\frac12\right)^k-
       \left(-i+\frac12\right)^k\right]
 +(1-2^{-k})\zeta(-k).
\end{equation}
The finite difference in this formula has finite support. Define
\begin{equation}\label{eq:bracket-intro}
 P(\rho)=\sum_\lambda\rho^{|\lambda|},\qquad
 \bracket{F}=P(\rho)^{-1}\sum_\lambda\rho^{|\lambda|}F(\lambda).
\end{equation}
For a series $\widehat R=\sum_{r\ge0}(-1)^r\h^{2r}R_r$, let
\begin{equation}\label{eq:times-intro}
 x_a=\h^a\left(\dd^{a+2}\widehat R-\delta_{a0}R_0''\right),
 \qquad
 \cC=\log\bracket{\exp\left(\sum_{a\ge0}
                   \frac{x_ap_{a+1}}{(a+1)!}\right)}.
\end{equation}
The subtraction in $x_0$ is essential.

\begin{theorem}[Strong completed cycle regularity]\label{thm:strong-intro}
Suppose
\[
 \widehat R=\Phi+A,\qquad A\in\CC[[t,\h^2]],\qquad
 \rho=t\,v(t),\quad v(0)\ne0.
\]
Set
\begin{equation}\label{eq:vacuum-intro}
 V=\sum_{j\ge1}b_j\h^{2j-2}\dd^{2j}\widehat R-b_1R_0''.
\end{equation}
Then the trace in \eqref{eq:times-intro} is coefficientwise well defined and
\begin{equation}\label{eq:strong-intro}
 \cC-V\in t\h^2\CC[[t,\h^2]].
\end{equation}
The same statement holds over any commutative $\QQ$-algebra, with $v(0)$ a
unit, using the normalized regular exponential in \cref{eq:normalized-K}.
\end{theorem}

Every nonempty partition contains a box of content zero. The factor $t$ in
\eqref{eq:strong-intro} comes from that elementary fact, after the entire
shifted-power-sum exponent has been resummed into a content product.
Individual terms in the original local/relative graph expansion need not
satisfy a gap. Moreover, all regular parts of the lower genus potentials are
retained in the resummation.

The proof uses shifted symmetric functions and partition traces
\cite{OO,BO,Okounkov,OP}, together with the content product formalism
\cite{HO}. Its essential additional ingredient is the choice of the singular
kernel $\cB$: a central difference converts its shifted-power-sum exponent
into a polynomial content product. This gives the strong regularity
\eqref{eq:strong-intro} while retaining every regular term of the input
potentials.

\subsection{Transfer and finite parts}
The geometric identity proved in \cref{prop:geometric-transfer} is
\begin{equation}\label{eq:geom-intro}
 U_g=(-1)^gR_g+\cC_g,\qquad
 \cC_g=[\h^{2g-2}]\cC.
\end{equation}
Each insertion in \eqref{eq:times-intro} has strictly positive coupling
degree, and $\cC_g$ depends only on $R_0,\ldots,R_{g-1}$. This triangularity
allows an induction without assuming the unknown genus $g$ gap.

\begin{theorem}[Gap transfer]\label{thm:transfer-intro}
Suppose $R_g,U_g$ are meromorphic germs for $g\ge2$ and satisfy \eqref{eq:geom-intro} and
\eqref{eq:times-intro}, the nome has a simple zero, and
\[
 R_0=f_0+\text{holomorphic},\qquad
 R_1=\frac1{24}\log t+\text{holomorphic}.
\]
Then the ordinary all genus gap $U_g=u_gt^{2-2g}+O(1)$ is equivalent to the
relative all genus gap $R_g=r_gt^{2-2g}+O(1)$.
\end{theorem}

The proof is genuswise: no assumption on arbitrarily high genera is
needed for a fixed genus conclusion. More precisely, once the lower
relative gaps hold, the excess principal parts of $U_g$ and $R_g$ agree
up to the sign $(-1)^g$; see \cref{prop:polar-defect}.

There is also information beyond the polar part. Define
\[
 \bR_0=R_0-f_0,\quad \bR_1=R_1-\frac1{24}\log t,\quad
 \bR_g=R_g-r_gt^{2-2g},\quad \bU_g=U_g-u_gt^{2-2g}\quad(g\ge2).
\]
Then
\begin{equation}\label{eq:finite-intro}
 \bU_g(0)-(-1)^g\bR_g(0)
 =\sum_{j=1}^g b_j(-1)^{g-j}\bR_{g-j}^{(2j)}(0).
\end{equation}
A finite partition algorithm determines further Taylor coefficients. Once the
gap has been established, the known holomorphic anomaly equation determines
all relative potentials without additional positive degree input. We give
an explicit triangular ambiguity formula in \cref{prop:reconstruction}.
This last consequence uses, rather than reproves, the finite generation
and anomaly theorems of \cite{BFGW}.

\subsection{Higher genus structures and related work}
The mathematical distinction between structural equations and boundary
conditions is particularly visible in the quintic theory. The
mixed spin-$P$ construction of Chang-Li-Li-Liu \cite{CLLL} and the
$N$-mixed-spin-$P$ theory of Chang-Guo-Li-Li \cite{CGLN} supply moduli
spaces, virtual cycles and localization relations. Chang-Guo-Li proved
polynomial structure \cite{CGLPoly} and the BCOV Feynman rule
\cite{CGLBCOV}. A different geometric route is developed by Shuai Guo,
Felix Janda and Yongbin Ruan in \emph{Structure of Higher Genus
Gromov-Witten Invariants of Quintic 3-folds}
(\href{https://arxiv.org/abs/1812.11908v1}{arXiv:1812.11908v1}).
Using logarithmic gauged linear sigma model moduli and a comparison of
virtual cycles, that work proves finite generation, the holomorphic
anomaly equation and orbifold regularity; it distinguishes these results
from the conifold gap. The present argument addresses the latter type
of boundary problem for a different, relative theory.

These constructions belong to the wider geometry of phases of the
gauged linear sigma model. Fan-Jarvis-Ruan's quantum singularity theory
\cite{FJR} and their mathematical gauged linear sigma model
\cite{FJRGLSM} provide the Landau-Ginzburg side of this picture.
Chiodo-Ruan's quintic correspondence \cite{CR} identifies the genus zero
Landau-Ginzburg and Calabi-Yau theories through symplectic transformations.
The comparison below instead concerns the change to a conifold
polarization within the fixed local/relative correspondence.

The exact geometric input here is \cite[Theorem~1.2]{BFGW}, together with
its finite generation and anomaly statements
\cite[Theorems~1.6 and~1.8]{BFGW}. That work explicitly asks how the
ordinary and relative gaps are related through the correspondence
\cite[Section~1.6]{BFGW}. The genus zero local/relative comparison belongs
to the broader picture of \cite{vGGR}, and quantum Riemann-Roch
\cite{CG} underlies the higher genus correction. Shifted power sums,
elliptic partition traces and content products are classical
\cite{OO,BO,Okounkov,OP,HO}. Our additional step is the strong regularity
of the trace for the singular kernel \eqref{eq:universal-model}, including
arbitrary regular perturbations. The Bernoulli convolution and post gap
ambiguity reconstruction are consequences, not independent replacements
for that regularity. The modular descriptions \cite{CI,ASYYZ} fix the
polarization, while \cite{LP,OPi} give related anomaly structures.

For the broader refined theory, Andrea Brini and Yannik Schuler's
\emph{Refined Gromov-Witten invariants}
(\href{https://arxiv.org/abs/2410.00118v1}{arXiv:2410.00118v1}) proves
leading conifold asymptotics and explicitly separates them from the full
gap in Remark~6.12. Related sheaf interpretations are developed in
\cite{Bousseau,BFGW}. Huai-Liang Chang, Shuai Guo, Lei You and Huaigong
Zhang's 2026 preprint \emph{Conifold gap for the quintic threefold}
(\href{https://arxiv.org/abs/2609.13921v1}{arXiv:2609.13921v1}) uses
mixed spin-$P$ geometry and a cone vertex Laplace description to prove
gaps for the quintic and related theories.

The spectral-curve approach has a different chain of inputs: the
Eynard-Orantin recursion \cite{EO}, the Bouchard-Klemm-Mari\~no-Pasquetti
remodeling proposal \cite{BKMP}, and the proof of the toric orbifold
remodeling correspondence by Fang-Liu-Zong \cite{FLZ}. In their 2026
preprint \emph{Conifold Gap Theorem for Topological Recursion}
(\href{https://arxiv.org/abs/2608.11960v1}{arXiv:2608.11960v1}),
Bohan Fang, Juping Chen, Linji Chen and Zhengyu Zong prove a gap for
conifold-polarized topological-recursion free energies of local analytic
families of toric mirror curves with a generic single-node degeneration.
Their result treats separating and nonseparating nodes and gives a
remainder jointly holomorphic in the transverse and spectator parameters
at every fixed genus $g\geq2$. This is a spectral-curve theorem: a
corresponding boundary statement for a geometric Gromov-Witten potential
also requires a compatible mirror-map and polarization comparison.
It should not be read as an unrestricted conifold-gap theorem for all
Calabi-Yau threefolds. Neither this recursion theorem nor the mixed
spin-$P$ result is substituted for the elliptic relative correspondence
here. Our ordinary input remains Brini's local-plane theorem, and the
new step is the completed cycle transfer to the relative sector.

Murad Alim and Noah Tischler's \emph{From Closed to Relative Higher Genus
Gromov-Witten Invariants via Resurgent Functions}
(\href{https://arxiv.org/abs/2609.09256v1}{arXiv:2609.09256v1}) relates
strong coupling expansions for the resolved conifold to relative
invariants of $\PP^1$. In the present paper the boundary variable is
the local plane vanishing period, and the coupling remains formal.
Accordingly, the conclusions concern the stated relative, or
Nekrasov-Shatashvili, gap; they do not establish the full two-parameter
refined gap, an orientifold gap, or convergence of a nonperturbative
partition function.

\Cref{sec:geometry} identifies the conifold correction with an elliptic
cumulant trace, including its period, modular weight and descendant
normalizations. \Cref{sec:formal} proves strong regularity and the
triangular transfer, and deduces \cref{thm:main}.
\Cref{sec:jets} treats finite parts, regular Taylor coefficients,
ambiguity reconstruction and exact examples. All numerical examples are
separated from the proof of the gap.

\section{Conifold correspondence and elliptic cumulants}\label{sec:geometry}

\subsection{Relative potentials and the vanishing period}
Relative theories and gluing constructions also occur in the study of
changes of target geometry. Li-Ruan \cite{LR} developed symplectic
surgery methods for Calabi-Yau threefold invariants; Chen-Li-Li-Zhao
\cite{CLLZ} treated Ruan's conjecture for a specified class of singular
symplectic flops. In a distinct orbifold setting, Li-Chen-Du
\cite{LCD} derived a gluing recursion for a linear sigma model of a
$\PP^1$-orbifold by virtual localization. These are related instances of geometric
comparison, with targets different from the cubic boundary theory here.
The relative moduli and the precise higher genus identity used below
are those of \cite{Li01,Li02,BFGW}.

Let $H$ be the class of a line in $\PP^2$. Write
$\overline{\mathcal M}_{g}(\PP^2/E,d)$ for the moduli space of connected
relative stable maps in class $dH$, with a single relative marking of
contact order $3d$, in the relative stable map theory of
\cite{Li01,Li02}. We use the conventions of
\cite[Sections~1.1--1.2]{BFGW}, and set
\begin{equation}\label{eq:enumerative-definition}
 N_{g,d}^{\PP^2/E}
 =\int_{[\overline{\mathcal M}_{g}(\PP^2/E,d)]^{\rm vir}}
          (-1)^g\lambda_g,\qquad
 N_{g,d}^{K_{\PP^2}}
 =\int_{[\overline{\mathcal M}_{g}(K_{\PP^2},d)]^{\rm vir}}1.
\end{equation}
Here $\lambda_g$ is the top Chern class of the Hodge bundle. The relative
virtual dimension is $g$; the local virtual dimension is zero. In positive
degree, maps to the local threefold are supported on its zero section,
so the local invariants are defined by the usual local obstruction theory.

The generating functions, including their classical and unstable terms,
are
\begin{align}\label{eq:R-large}
 F_g^{\PP^2/E}(Q)
 &= -\delta_{g0}\frac{(\log Q)^3}{18}
    -\delta_{g1}\frac{\log Q}{24}
    +\sum_{d\ge1}\frac{(-1)^{3d+g-1}}{3d}
                   N_{g,d}^{\PP^2/E}Q^d,\notag\\
 F_g^{K_{\PP^2}}(Q)
 &= -\delta_{g0}\frac{(\log Q)^3}{18}
    -\delta_{g1}\frac{\log Q}{12}
    +\sum_{d\ge1}N_{g,d}^{K_{\PP^2}}Q^d.
\end{align}
No degree zero constant map terms are added in genus at least two.
At genus zero the two functions agree by the local/relative
correspondence. The mirror coordinate $q$ is related to $Q$ by
\begin{equation}\label{eq:mirror-map-definition}
 Q=q\exp\!\left(3\sum_{d\ge1}
       \frac{(3d-1)!}{(d!)^3}(-1)^dq^d\right),\qquad
 I_{11}=q\partial_q\log Q.
\end{equation}
We write $R_g=(F_g^{\PP^2/E})^c$ and $U_g=(F_g^{K_{\PP^2}})^c$,
where the superscript $c$ means the holomorphic limit polarization
specified below.

The analytic function $\phi$ in \eqref{eq:coord-intro} satisfies
\begin{equation}\label{eq:PF}
 \xi(1-\xi)\phi''+(1-2\xi)\phi'-\frac29\phi=0.
\end{equation}
The vanishing period has expansion
\begin{equation}\label{eq:T-expansion}
 T:=\sqrt3t=\xi+\frac{11}{18}\xi^2+\frac{109}{243}\xi^3
       +\frac{9389}{26244}\xi^4+O(\xi^5).
\end{equation}
In particular, it is an analytic coordinate. The coefficient of its linear
term fixes the scale relevant to the gap.

\begin{lemma}[Uniqueness of the normalized vanishing period]\label[lemma]{lem:period-scale}
The analytic solution of the local plane Picard-Fuchs equation which
vanishes at $\xi=0$ and has derivative $1$ there is $T=\sqrt3t$.
Consequently Brini's normalized vanishing period $t_{\rm CF}$ is exactly
$\sqrt3t$, not merely equal to it to first order.
\end{lemma}
\begin{proof}
Write the local plane period operator as
\[
 \mathcal D=\theta^3+3q\theta(3\theta+1)(3\theta+2),
 \qquad \theta=(\xi-1)\partial_\xi.
\]
For $\psi=(1-\xi)\partial_\xi Y$, direct differentiation gives
\[
 \mathcal D Y=(\xi-1)
 \left(\xi(1-\xi)\psi''+(1-2\xi)\psi'-\frac29\psi\right).
\]
Since $\xi-1$ is a unit at the conifold, an analytic period has
$\psi$ satisfying \eqref{eq:PF}. The coefficient
recurrence for an analytic solution $\psi=\sum_{n\ge0}a_n\xi^n$ is
\[
 (n+1)^2a_{n+1}=(n+\tfrac13)(n+\tfrac23)a_n,\qquad n\ge0.
\]
Thus $\psi(0)=1$ determines $\psi=\phi$ to all orders. Integration with
$Y(0)=0$ gives precisely $T$ in \eqref{eq:T-expansion}.
Brini's $t_{\rm CF}$ has the same defining Picard-Fuchs equation and
normalization, so uniqueness proves the equality.
\end{proof}

Set
\[
 X=(1+27q)^{-1},\qquad
 S=\theta\log I_{11}-\frac{X-1}{3},\qquad \theta=q\partial_q.
\]
The differential ring variables of \cite{BFGW} satisfy
\begin{equation}\label{eq:ring}
 \theta X=X(X-1),\qquad
 \theta S=-S^2+\frac{X-1}{3}S-\frac{X(X-1)}9,
 \qquad \theta=q\partial_q.
\end{equation}
Their conifold evaluation is
\begin{equation}\label{eq:evaluation}
 I_{11}^c=\phi,\qquad X^c=\xi^{-1},\qquad
 S^c=(\xi-1)\frac{\phi'}\phi-\frac{\xi^{-1}-1}{3}.
\end{equation}
We write $X,S$ for the evaluated germs when no ambiguity is possible. Substitution of \eqref{eq:PF} proves \eqref{eq:ring} for
these germs. The remaining generator relation is
\begin{equation}\label{eq:I11-derivative}
 \theta I_{11}=I_{11}\left(S+\frac{X-1}{3}\right).
\end{equation}
It evaluates to $\theta\phi=(\xi-1)\phi'$, by
\eqref{eq:evaluation}. Thus the evaluation is compatible with all three
generator relations, not merely with a finite number of expansions. Since
$\theta=(\xi-1)\partial_\xi$, the operator
$D=3I_{11}^{-1}\theta$ becomes
\begin{equation}\label{eq:D-scale}
 D^c=3\phi^{-1}(\xi-1)\partial_\xi=-\sqrt3\,\dd.
\end{equation}
This holds for all iterated derivatives of the evaluated expressions.

\subsection{The normalized cusp action}
We spell out the modular normalization because its signs will cancel the
signs in the geometric correspondence. Use the generators
\begin{align}\label{eq:modgens}
 \mA(\tau)^3&=\frac{\eta(\tau)^9}{\eta(3\tau)^3}
               +27\frac{\eta(3\tau)^9}{\eta(\tau)^3},\notag\\
 \mB(\tau)&=\frac{E_2(\tau)+3E_2(3\tau)}4,\qquad
 \mC(\tau)=\frac{\eta(\tau)^9}{\eta(3\tau)^3}.
\end{align}
Their weights are $1,2,3$, respectively, and
\begin{equation}\label{eq:ring-mod}
 I_{11}=\mA,\qquad X=\frac{\mA^3}{\mC},\qquad
 S=\frac{\mA\mB-\mA^3}{6\mC}.
\end{equation}
These are the generators in \cite[Section~4]{BFGW}; for modular changes of
polarization see also \cite{ABK,ASYYZ,CI}.

Write $z=-1/(3\tau)$ and $\rho=e^{2\pi iz}$. In making the cusp change, first
replace $E_2$ by its almost holomorphic completion, apply the weight
transformation, and then take the holomorphic part at the $z$-cusp. This
prescription is essential for $\mB$. If $\mu=-i\sqrt3z$, the eta
transformation and the weight two transformation of the completion give
\begin{align}\label{eq:Fricke-raw}
 \mA\left(-\frac1{3z}\right)&=\mu\,\mA(z),\notag\\
 \widehat{\mB}\left(-\frac1{3z}\right)&=-\mu^2\widehat{\mB}(z),\notag\\
 \mC\left(-\frac1{3z}\right)&=
       \mu^3\bigl(\mA(z)^3-\mC(z)\bigr).
\end{align}
The root for $\mA$ is fixed by continuation of the normalized period.
After division by the weight factor, the normalized cusp action is
\begin{equation}\label{eq:Fricke}
 \mA^c=\mA(\rho),\qquad \mB^c=-\mB(\rho),\qquad
 \mC^c=\mA(\rho)^3-\mC(\rho).
\end{equation}
The last term has a simple zero at $\rho=0$; it is not to be replaced by
the large radius series $\mC(\rho)$ itself. The potentials at genus at least
two have weight zero, so the common weight trivialization cancels.

The relation between the nome and $\xi$ is equivalently specified by
\begin{equation}\label{eq:nome}
 \frac{d\log\rho}{d\xi}=\frac1{\xi(1-\xi)\phi(\xi)^2},\qquad
 \rho=\frac{\xi}{27}+O(\xi^2).
\end{equation}
Then $\mA(\rho)=\phi$ and $\mC^c=\xi\phi^3$. Together with
\eqref{eq:ring-mod}, this recovers \eqref{eq:evaluation}.
In particular,
\begin{equation}\label{eq:nome-linear}
 \rho=\rho_1t+O(t^2),\qquad \rho_1=\frac{\sqrt3}{27}.
\end{equation}

\begin{remark}[Holomorphic cusp representatives]\label[remark]{rem:cusp-meaning}
The superscript $c$ denotes the holomorphic part at the conifold cusp of
the same almost holomorphic modular function, as in
\cite[Section~1.6]{BFGW}. Equations
\eqref{eq:Fricke-raw}--\eqref{eq:Fricke} express this operation in the
large radius generators. In particular, it is not evaluation of the
unchanged Fourier series $\mA,\mB,\mC$ at $\rho$: that would leave the
denominator nonzero at $\rho=0$. The transformed denominator
$\mC^c=\xi\phi^3$ vanishes simply. This convention also agrees with the
conifold opposite description in \cite[Sections~10.3 and~10.8]{CI}.
\end{remark}

\begin{lemma}[Elliptic weight transformation]\label[lemma]{lem:elliptic-weight}
For $k=1,2,3$, the normalized conifold evaluation satisfies
\begin{equation}\label{eq:E-scale}
 \bigl(E_{2k}(3\tau)\bigr)^c=\left(-\frac13\right)^kE_{2k}(\rho).
\end{equation}
Consequently, a homogeneous elliptic quasimodular form of weight $2w$
acquires the factor $(-1/3)^w$.
\end{lemma}
\begin{proof}
The embedding of elliptic quasimodular forms in the local plane ring is
\cite[Proposition~4.8]{BFGW}
\begin{align}\label{eq:E-embed}
 3E_2(3\tau)&=2\mB+\mA^2,\notag\\
 9E_4(3\tau)&=\mA^4+8\mA\mC,\notag\\
 27E_6(3\tau)&=-\mA^6+20\mA^3\mC+8\mC^2.
\end{align}
The same formulas in the $z$ variable, together with
$3E_2(3z)-E_2(z)=2\mA(z)^2$, imply the assertion using
\eqref{eq:Fricke}. Alternatively, define
\begin{align}\label{eq:eis-PF}
 e_2&=\phi^2(1-4\xi)+12\xi(1-\xi)\phi\phi',\notag\\
 e_4&=\phi^4(1+8\xi),\qquad
 e_6=\phi^6(1-20\xi-8\xi^2).
\end{align}
A substitution of \eqref{eq:PF}, with
$\delta=\xi(1-\xi)\phi^2\partial_\xi=\rho\partial_\rho$, gives
\[
 \delta e_2=\frac{e_2^2-e_4}{12},\qquad
 \delta e_4=\frac{e_2e_4-e_6}{3},\qquad
 \delta e_6=\frac{e_2e_6-e_4^2}{2}.
\]
The normalization \eqref{eq:nome} gives the first terms
$e_2=1-24\rho+O(\rho^2)$, $e_4=1+240\rho+O(\rho^2)$ and
$e_6=1-504\rho+O(\rho^2)$. These determine the entire solution. Indeed,
the linearization of the displayed system at $(1,1,1)$ is
\[
 J=\begin{pmatrix}
  1/6&-1/12&0\\1/3&1/3&-1/3\\1/2&-1&1/2
 \end{pmatrix},\qquad \det(nI-J)=n^2(n-1).
\]
After the coefficient of $\rho$ is fixed, the coefficient equations are
invertible for every $n\ge2$. Hence $e_{2k}=E_{2k}(\rho)$ to all orders.
Substitution in \eqref{eq:E-embed} gives \eqref{eq:E-scale}. Since
$E_2,E_4,E_6$ generate the elliptic quasimodular ring, homogeneity proves
the last statement.
\end{proof}

\begin{remark}
Equation \eqref{eq:E-scale} includes the holomorphic limit change for
$E_2$. A naive argument substitution in an uncompleted quasimodular form
would leave an inhomogeneous term and would not justify the lemma.
The differential ring calculation above supplies an independent way to
verify the required convention.
\end{remark}

\subsection{Low genus and meromorphicity}
The genus zero identity and genus one formula in
\cite[Lemma~4.9 and Theorem~1.4]{BFGW} imply
\begin{equation}\label{eq:R0-R1}
 R_0'''=\frac{\sqrt3}{\xi\phi^3},\qquad
 R_1=\frac1{24}\log\frac{\xi}{1-\xi}+\text{constant}.
\end{equation}
The first equality follows from $D^3F_0=-9X/I_{11}^3$ and
\eqref{eq:D-scale}. Since the difference of its right hand side and $1/t$
is holomorphic, integration gives
\begin{equation}\label{eq:low}
 R_0=f_0+\bR_0,\qquad R_1=\frac1{24}\log t+\bR_1,
 \qquad \bR_0,\bR_1\in\CC\{t\}.
\end{equation}
The quadratic ambiguity in $R_0$ is immaterial: only third and higher
derivatives of $R_0$ occur in the correction.
For an independent check on the ordinary polarization, evaluate the known
local genus one expression as
\[
 U_1=-\frac1{12}\log\bigl(\xi(1-\xi)\bigr)
       -\frac12\log\phi+\text{constant}.
\]
Its derivative in $T=\sqrt3t$ is
\begin{equation}\label{eq:ordinary-one-check}
 \frac{dU_1}{dT}=-\frac1{12T}+\frac5{216}-\frac{T}{11664}+O(T^2).
\end{equation}
These terms agree with the conifold opposite computation in
\cite[Section~10.8]{CI}. This is a normalization check, not a substitute
for the all order holomorphic limit identification above.

Finite generation \cite[Theorem~1.6]{BFGW} gives
\[
 F_g^{\PP^2/E}\in\mC^{-(2g-2)}\QQ[\mA,\mB,\mC]_{6g-6},\qquad g\ge2.
\]
Since $\mC^c$ has a simple zero and $t$ is analytic, $R_g$ is a meromorphic
germ with pole order at most $2g-2$. The ordinary potentials have the
corresponding fixed genus meromorphic description
\cite{CI,LP}. This statement is sufficient to turn the formal cancellation
of the principal part into holomorphicity of the fixed genus remainder.

\subsection{The elliptic variable and the local/relative correction}
At large radius the elliptic degree variable in the comparison is
\begin{equation}\label{eq:elliptic-degree-variable}
 \widetilde{\mathcal Q}
 =-\exp\bigl(-D^2F_0^{K_{\PP^2}}\bigr)
 =e^{2\pi i(3\tau)},\qquad D=3Q\partial_Q.
\end{equation}
The first equality is \cite[Equation~(7)]{BFGW}; the second follows from
\cite[Lemma~4.9 and Section~4.6]{BFGW}. Thus the stationary elliptic
potentials are quasimodular forms in $3\tau$ before the cusp change.
They are transformed using \cref{lem:elliptic-weight}. In particular, one
must transform their modular representatives, rather than substitute a
conifold genus zero primitive into the exponential in
\eqref{eq:elliptic-degree-variable}. The proof will only require that
$\rho$ has the simple zero \eqref{eq:nome-linear}.

Let $F^E_{\ell,\mathbf a}(\rho)$ denote the connected stationary elliptic
potential with insertions $\tau_{a_j}(\omega)$, where $\omega$ is the point
class. Its dimension constraint is
$\sum_ja_j=2\ell-2$, and its quasimodular weight is
$\sum_j(a_j+2)=2\ell-2+2n$. Its disconnected convention excludes unmarked
components; connected marked degree zero terms are retained. These are
the conventions used in \cite[Section~3.4]{BFGW}.

\begin{proposition}[Conifold local/relative identity]\label[proposition]{prop:local-relative}
For every $g\ge2$ one has $U_g=(-1)^gR_g+\cC_g$, where
\begin{equation}\label{eq:graphsum}
 \cC_g=\sum_{n\ge1}\frac1{n!}
 \sum_{\substack{g=\ell+\sum_jr_j,\ \sum_ja_j=2\ell-2\\
                  a_j,r_j\ge0,\ (a_j,r_j)\ne(0,0)}}
 F^E_{\ell,\mathbf a}(\rho)
 \prod_{j=1}^n(-1)^{r_j}\dd^{a_j+2}R_{r_j}.
\end{equation}
The inner sum is over ordered decorations.
\end{proposition}
\begin{proof}
Start with \cite[Theorem~1.2]{BFGW}. Before conifold evaluation, the elliptic
factor has sign $(-1)^{\ell-1}$ and each relative factor has sign
$(-1)^{r_j-1}$. By \cref{lem:elliptic-weight}, its weight gives an additional
$(-1/3)^{\ell-1+n}$. The derivative factors in \eqref{eq:D-scale} give
\[
 (-\sqrt3)^{\sum_j(a_j+2)}=3^{\ell-1+n},
\]
since the exponent is even. The powers of three cancel and the total sign is
\[
 (-1)^{\ell-1+\sum_jr_j-n+\ell-1+n}=(-1)^{\sum_jr_j}.
\]
For a multiset of decorations with multiplicities $m_{a,r}$, replacing the
unordered sum by the ordered sum contributes
\[
 \frac1{n!}\frac{n!}{\prod_{a,r}m_{a,r}!}
 =\frac1{\prod_{a,r}m_{a,r}!}=\frac1{|\Aut|}.
\]
Finally, $n=0$ forces $\ell=g=1$ by the dimension constraint, and hence
contributes nothing for $g\ge2$. This proves \eqref{eq:graphsum}.
\end{proof}

\subsection{Completed cycles and connected traces}
The shift constants in \eqref{eq:shifted-intro} are part of the completed
power sums. In particular,
\begin{equation}\label{eq:vac-constants}
 p_1(\lambda)=|\lambda|-\frac1{24},\qquad
 p_{2j}(\varnothing)=0,\qquad
 \frac{p_{2j-1}(\varnothing)}{(2j-1)!}=b_j.
\end{equation}
The last identity follows from $\zeta(1-2j)=-B_{2j}/(2j)$ and
$B_{2j}(1/2)=(2^{1-2j}-1)B_{2j}$.

\begin{lemma}[Descendant normalization and cumulants]\label[lemma]{lem:OP}
The elliptic connected stationary potential is the joint cumulant, under
\eqref{eq:bracket-intro}, of the observables
$p_{a_j+1}/(a_j+1)!$. Equivalently,
\begin{equation}\label{eq:elliptic-log}
 \log\bracket{\exp\left(\sum_{a\ge0}
                  y_a\frac{p_{a+1}}{(a+1)!}\right)}
 =\sum_{n\ge1}\frac1{n!}\sum_{a_1,\ldots,a_n\ge0}
   F^E_{\ell,\mathbf a}(\rho)\prod_{j=1}^ny_{a_j},
\end{equation}
where the summand is zero unless $\sum a_j=2\ell-2$ for an integer
$\ell\ge1$.
\end{lemma}
\begin{proof}
The completed cycle correspondence of \cite[Equation~(0.23)]{OP} is
$\tau_a(\omega)=\overline{(a+1)}/a!$. The eigenvalue of
$\overline{(a+1)}$ is $p_{a+1}/(a+1)$, giving precisely the factorial
$(a+1)!$ in \eqref{eq:elliptic-log}. In the character formula the
representation dimension factor is raised to the Euler characteristic of
the target. It is therefore absent for an elliptic target. The disconnected
trace is a sum over partitions with weight $\rho^{|\lambda|}$, as in
\cite[Section~5]{OP}. Dividing by $P(\rho)$ removes the unmarked partition
factor; taking the logarithm extracts connected contributions.

The dimension constraint for a connected genus $\ell$ stationary invariant
is $\sum a_j=2\ell-2$. The marked degree zero one-point terms survive and
are exactly the constants in \eqref{eq:vac-constants}. Connected marked
degree zero terms with more than one marking vanish. Thus normalization by
$P(\rho)$ does not delete the required vacuum constants, and the connected
formula is \eqref{eq:elliptic-log}.
\end{proof}

\begin{proposition}[Exact geometric cumulant identity]\label[proposition]{prop:geometric-transfer}
With the times in \eqref{eq:times-intro},
\begin{equation}\label{eq:exact-C}
 \sum_{g\ge2}\h^{2g-2}\cC_g
 =\log\bracket{\exp\left(\sum_{a\ge0}
                   \frac{x_ap_{a+1}}{(a+1)!}\right)}.
\end{equation}
For fixed $g$, its coefficient depends only on $R_0,\ldots,R_{g-1}$.
\end{proposition}
\begin{proof}
Expand each time as
\[
 x_a=\sum_{\substack{r\ge0\\(a,r)\ne(0,0)}}
       (-1)^r\h^{a+2r}\dd^{a+2}R_r.
\]
An ordered selection of insertions has coupling degree
$\sum a_j+2\sum r_j=2\ell-2+2\sum r_j=2g-2$.
Thus \cref{lem:OP} reproduces \eqref{eq:graphsum}, including its signs and
factorials. Each allowed decoration has $a+2r>0$, so only finitely many
insertions occur at a fixed coupling degree. Moreover, $a+2r\le2g-2$
forces $r\le g-1$. This proves the triangularity.
\end{proof}

\section{Content regularity and the all genus transfer}\label{sec:formal}

\subsection{Coefficientwise meaning of the partition sum}
We first work over a commutative $\QQ$-algebra $K$; the geometric application
has $K=\CC$. The variable $t$ is independent of the formal coupling $\h$.
The intermediate Laurent series belong to the algebra
\begin{equation}\label{eq:Laurent-algebra}
 \cL_K=\left\{\sum_{m\ge0}\h^m f_m(t):
                    f_m(t)\in t^{-m}K[[t]]\right\}.
\end{equation}
To specify its topology, put
$F^n\cL_K=\{\sum_{m\ge n}\h^mf_m(t)\}$. With an auxiliary variable $z$,
the map
\begin{equation}\label{eq:completion-isomorphism}
 \sum_{m\ge0}\h^mt^{-m}a_m(t)\longmapsto
 \sum_{m\ge0}z^ma_m(t)
\end{equation}
identifies $\cL_K$ with $K[[t]][[z]]$ and $F^n$ with the $z$-adic
filtration. Thus the algebra is complete, and exponentiation and
logarithms on $F^1\cL_K$ and $1+F^1\cL_K$, respectively, are well defined.
This filtration is not the principal ideal $\h$-adic filtration:
$\h/t$ lies in $F^1$ but need not lie in $\h\cL_K$.
We also use the jointly complete ring $K[[t,\h]]$ for the resummed
expressions; its embedding in \eqref{eq:completion-isomorphism} sends
$\h$ to $tz$. The auxiliary symbol $\log t$ has its usual derivative;
no unexpanded logarithm remains in the times or their partition weights.

Write $A=\sum_{r\ge0}\h^{2r}A_r(t)$ and assume
\begin{equation}\label{eq:formal-input}
 \widehat R=\Phi+A,\qquad A\in K[[t,\h^2]],\qquad
 \rho=t\,v(t),\quad v(t)\in K[[t]]^\times.
\end{equation}
For $a\ge1$, the contribution of $f_0$ to $x_a$ has pole order $a$;
for $r\ge1$, the contribution to coupling degree $a+2r$ has pole order at
most $a+2r$. The $(a,r)=(0,0)$ logarithmic term is removed. Hence all
$x_a$ have strictly positive coupling order and satisfy the bound defining
\eqref{eq:Laurent-algebra}.

\begin{lemma}[Finite coefficient bound]\label[lemma]{lem:cutoff}
At coupling degree $m$, a partition of size $d$ in the non-vacuum trace
has $t$-valuation at least $d-m$. Therefore a coefficient $t^k\h^m$
receives contributions only from partitions of size $d\le k+m$.
\end{lemma}
\begin{proof}
At degree $m$, a product of times has pole order at most $m$. The factor
$\rho^d=t^dv(t)^d$ raises the valuation by $d$. The stated bound follows.
There are finitely many partitions with bounded size. For a fixed $m$,
the remaining sum is thus a well defined Laurent series in $t$.
\end{proof}

This lemma is also the justification for later regrouping. We do not
interchange divergent numerical sums: the two expansions are compared on
each coefficient, for which only finitely many partitions occur.

\subsection{The singular kernel and its central difference}
The first terms of the universal model are
\begin{equation}\label{eq:Phi-explicit}
 \Phi=f_0-\frac{\h^2}{24}\log t
          -\frac{7\h^4}{5760t^2}
          +\frac{31\h^6}{161280t^4}+O(\h^8).
\end{equation}
The signs are chosen for $\widehat R=\sum(-1)^r\h^{2r}R_r$.

\begin{lemma}[Central difference identity]\label[lemma]{lem:central}
In the formal Taylor sense,
\begin{equation}\label{eq:central}
 \Phi'\left(t+\frac\h2,\h\right)
 -\Phi'\left(t-\frac\h2,\h\right)=\h\log t.
\end{equation}
\end{lemma}
\begin{proof}
The operator $\cB(\h\dd)$ is a formal power series in $\h\dd$ with
constant term one. Multiplication by $2\sinh(\h\dd/2)$ gives
\[
 2\sinh\left(\frac{\h\dd}{2}\right)\cB(\h\dd)\dd f_0
 =\h\dd^2f_0=\h\log t.
\]
The left side is precisely the difference in \eqref{eq:central}.
Each coupling coefficient is a finite differential expression, so the
operator computation is coefficientwise valid.
\end{proof}

Difference operators provide another useful description of singular
Gromov-Witten potentials. Alim \cite{AlimDiff} proves a difference
equation for the resolved conifold and its genus zero
Gopakumar-Vafa contribution. The author's \emph{Difference Equations
and Reconstruction of Local Gopakumar-Vafa Spectra for Threefold Flops}
(\href{https://arxiv.org/abs/2609.10488v2}{arXiv:2609.10488v2}) studies
central difference forcing and recovery of finite local spectra from
multiple cover data. Equation~\eqref{eq:central} serves a different
purpose: its right hand side fixes a singular kernel whose rowwise
differences resum the actual elliptic correction. Unlike reconstruction from finite multiple cover data, the problem
here is cancellation in the complete elliptic correction, including
all regular lower genus terms.

For a partition $\lambda$, put
\begin{equation}\label{eq:rowcoords}
 u_i=\lambda_i-i+\frac12,\qquad v_i=-i+\frac12.
\end{equation}
These $v_i$ are row coordinates, unrelated to the nome unit $v(t)$.

\begin{lemma}[Content identity]\label[lemma]{lem:content}
For every partition $\lambda$,
\begin{equation}\label{eq:content}
 \exp\left\{\frac1\h\sum_i
       \bigl[\Phi'(t+\h u_i,\h)-\Phi'(t+\h v_i,\h)\bigr]\right\}
 =\prod_{\square\in\lambda}(t+\h c(\square)).
\end{equation}
The exponential on the left is interpreted after extracting its constant
term $|\lambda|\log t$.
\end{lemma}
\begin{proof}
For a fixed row $i$, apply \eqref{eq:central} with center
$t+\h(j-i)$ for $j=1,\ldots,\lambda_i$. Adjacent endpoint terms cancel,
leaving
\[
 \frac1\h\bigl[\Phi'(t+\h u_i)-\Phi'(t+\h v_i)\bigr]
 =\sum_{j=1}^{\lambda_i}\log(t+\h(j-i)).
\]
The logarithm on the right means
$\log t+\log(1+\h(j-i)/t)$. Exponentiation of this finite sum, followed by
multiplication over rows, proves \eqref{eq:content}.
\end{proof}

For example, the three boxes of $\lambda=(2,1)$ have contents $0,1,-1$.
The right side is $t(t+\h)(t-\h)=t^3-t\h^2$. Terms with many apparent
negative powers in the unresummed exponential have combined into a
polynomial. This elementary example illustrates why termwise regularity
of the original graph sum is not required.

Young diagram sums are also fundamental to the topological vertex
\cite{AKMV} and its refined extension \cite{IKV}. Kim-Li-Yagi-Zhu
\cite{KLYZ} construct a vertex formalism for symmetric matter in
five dimensional gauge theory, including an $\mathrm O7^+$-plane
configuration, whose partition functions are likewise diagram sums.
The common feature is diagrammatic reorganization. In the present
argument, its geometry is fixed by the stationary elliptic completed cycle
trace: the vacuum constants and descendant normalization are those of
\cref{lem:OP}, and the content identity acts on precisely that trace.

\subsection{Vacuum separation and the regular exponential}
For each partition, write
\begin{equation}\label{eq:vac-separate}
 \sum_{a\ge0}\frac{x_ap_{a+1}(\lambda)}{(a+1)!}=V+B_\lambda,
\end{equation}
where $V$ is the value at the empty partition. From
\eqref{eq:vac-constants} one obtains exactly
\begin{equation}\label{eq:V-formula}
 V=\sum_{j\ge1}b_j\h^{2j-2}\dd^{2j}\widehat R-b_1R_0''.
\end{equation}
In particular $V$ has positive coupling order, despite the individual
$\log t$ in its first summand.
Taylor expansion of the finite row differences gives
\begin{equation}\label{eq:B-lambda}
 B_\lambda=\frac1\h\sum_i
 [\widehat R'(t+\h u_i,\h)-\widehat R'(t+\h v_i,\h)]
 -|\lambda|R_0''.
\end{equation}
The second term occurs because the missing $(a,r)=(0,0)$ insertion
multiplies $p_1(\lambda)-p_1(\varnothing)=|\lambda|$.

It is useful to remove the entire coupling constant term from the regular
exponent. Define
\begin{equation}\label{eq:normalized-K}
 \widetilde K_\lambda(t,\h)=\frac1\h\sum_i
 [A'(t+\h u_i,\h)-A'(t+\h v_i,\h)]-|\lambda|A_0''(t).
\end{equation}
The numerator is divisible by $\h$, and its quotient at $\h=0$ is
$|\lambda|A_0''(t)$. Consequently
\begin{equation}\label{eq:K-positive}
 \widetilde K_\lambda\in\h K[[t,\h]].
\end{equation}
The exponential $\exp\widetilde K_\lambda$ is therefore defined over any
$\QQ$-algebra; no exponential of an arbitrary coefficient in $K$ is needed.

Combining \eqref{eq:content} and \eqref{eq:B-lambda} gives the exact weight
identity
\begin{equation}\label{eq:weight-content}
 \rho^{|\lambda|}e^{B_\lambda}
 =v(t)^{|\lambda|}
   \prod_{\square\in\lambda}(t+\h c(\square))
   e^{\widetilde K_\lambda(t,\h)}.
\end{equation}
Notice that no equality between $\rho$ and $e^{R_0''}$ has been imposed.
The independent unit $v(t)$ is visible on the right.

\begin{theorem}[Strong regularity, coefficient form]\label{thm:strong}
Under \eqref{eq:formal-input}, define $\cC$ by \eqref{eq:times-intro} and
$V$ by \eqref{eq:V-formula}. Then
\begin{equation}\label{eq:strong}
 \cN:=\cC-V\in t\h^2K[[t,\h^2]].
\end{equation}
\end{theorem}
\begin{proof}
Let $Z_A$ be the sum over all partitions of the right side of
\eqref{eq:weight-content}. The content polynomial for a size $d$ partition
is homogeneous of total $(t,\h)$-degree $d$. Its other factors have
nonnegative total degree. The partition sum is therefore well defined in
$K[[t,\h]]$. Every nonempty partition contains $(1,1)$, a box of content
zero. Hence
\[
 Z_A\in1+tK[[t,\h]],\qquad P(\rho)\in1+tK[[t]].
\]
Vacuum separation and \cref{lem:cutoff} identify the original trace with
this resummed expression, and give
\begin{equation}\label{eq:logsplit}
 \cC=V+\log Z_A-\log P(\rho).
\end{equation}
The two logarithms in \eqref{eq:logsplit} are defined in the
$t$-adically complete subring $K[[\h]][[t]]$, and lie in
$tK[[t,\h]]$. Since $\widetilde K_\lambda(t,0)=0$, the weight at $\h=0$
is $v(t)^dt^d=\rho^d$, so $Z_A(t,0)=P(\rho)$. The quotient
$Z_A/P(\rho)$ therefore has constant coupling coefficient one.
The identity
$\log Z_A-\log P(\rho)=\log(Z_A/P(\rho))$ in the commutative
joint power series ring agrees, coefficientwise, with the $F^\bullet$
logarithm of that quotient in $\cL_K$. The finite bound of
\cref{lem:cutoff} justifies the comparison. Hence the difference lies
also in $\h K[[t,\h]]$, with no exchange of numerical limits.

It remains to prove evenness. Partition conjugation satisfies
\[
 p_k(\lambda')=(-1)^{k-1}p_k(\lambda).
\]
This includes the constants, since the even index vacuum constants vanish.
One may verify the nonconstant part from
\[
 \sum_i(e^{zu_i}-e^{zv_i})
 =2\sinh(z/2)\sum_{\square\in\lambda}e^{zc(\square)};
\]
conjugation reverses all contents. Also
$x_a(t,-\h)=(-1)^a x_a(t,\h)$. Pairing $\lambda$ and $\lambda'$ proves
that the full trace, hence its logarithm, is even in $\h$. Formula
\eqref{eq:V-formula} shows the same for $V$. The logarithmic difference
therefore belongs to $t\h^2K[[t,\h^2]]$.
\end{proof}

\begin{remark}[Regular perturbations and parameter dependence]
The theorem keeps every coefficient of $A$; it is not a statement obtained
by discarding the regular part of each $R_r$. Such a discard would generally
be invalid before resummation. Since the normalized exponential
\eqref{eq:normalized-K} has zero coupling constant term, the proof also
works with formal auxiliary parameters in $K$. This algebraic fact is not,
by itself, a geometric theorem about multiparameter degenerations.
\end{remark}

\subsection{The vacuum convolution}
For $g\ge2$,
\begin{equation}\label{eq:f0derivative}
 \dd^{2g}f_0=-(2g-3)!\,t^{2-2g}.
\end{equation}
The singular part of $[\h^{2r}]\Phi$ equals $(-1)^rr_rt^{2-2r}$ for
$r\ge2$, and the first two terms match \eqref{eq:low} with the chosen
signs. By \cref{thm:strong}, derivatives of the regular perturbation cannot
contribute to the principal part of $\cC_g$. Thus
\begin{equation}\label{eq:vac-conv}
 \Pol\cC_g=-(2g-3)!\sum_{j=1}^gb_jb_{g-j}\,t^{2-2g}.
\end{equation}

\begin{lemma}[Bernoulli convolution]\label[lemma]{lem:Bernoulli}
For every $g\ge2$,
\begin{equation}\label{eq:gamma}
 -(2g-3)!\sum_{j=1}^gb_jb_{g-j}
 =\gamma_g:=u_g-(-1)^gr_g.
\end{equation}
\end{lemma}
\begin{proof}
The identity
\begin{equation}\label{eq:B-square}
 \cB(z)^2=\frac{z^2e^z}{(e^z-1)^2}
 =\sum_{g\ge0}(1-2g)\frac{B_{2g}}{(2g)!}z^{2g}
\end{equation}
follows by subtracting $z$ times the derivative of $z/(e^z-1)$ from
$z/(e^z-1)$. The odd linear term cancels. Consequently
\[
 -(2g-3)!\sum_{j=0}^gb_jb_{g-j}
 =\frac{B_{2g}}{2g(2g-2)}=u_g.
\]
The $j=0$ term is $-(2g-3)!b_g=(-1)^gr_g$, because
$B_{2g}(1/2)=(2^{1-2g}-1)B_{2g}$ and
$B_{2g}=(-1)^{g-1}|B_{2g}|$. Removing it proves \eqref{eq:gamma}.
\end{proof}

For reference, the first values are displayed in \cref{tab:coefficients}.
They are consequences of the convolution, not independent enumerations.
\begin{table}[ht]
\centering
\caption{Universal relative, correction, and ordinary polar coefficients.}
\label{tab:coefficients}
\renewcommand{\arraystretch}{1.7}
\begin{tabular}{c c c c}\toprule
$g$ & $r_g$ & $\gamma_g$ & $u_g$\\\midrule
2 & $-7/5760$ & $-17/5760$ & $-1/240$\\
3 & $-31/161280$ & $43/53760$ & $1/1008$\\
4 & $-127/1290240$ & $-769/1290240$ & $-1/1440$\\\bottomrule
\end{tabular}
\end{table}

\subsection{Polar defects and proof of the gap transfer}
For a meromorphic germ, let $\Pol$ denote its full principal part. Define
\begin{equation}\label{eq:polar-defects}
 \Delta_g^U=\Pol U_g-u_gt^{2-2g},\qquad
 \Delta_g^R=\Pol R_g-r_gt^{2-2g}.
\end{equation}
The subtraction removes the expected monomial; the remaining Laurent
polynomial records an incorrect leading coefficient as well as any
subleading polar terms.

\begin{proposition}[Single genus polar defect identity]\label[proposition]{prop:polar-defect}
Fix $g\ge2$. Assume the low genus conditions and exact correspondence
of \cref{thm:transfer-intro}, and assume the relative gaps only in
 genera $2\le r<g$. Then
\begin{equation}\label{eq:polar-defect-identity}
 \Delta_g^U=(-1)^g\Delta_g^R.
\end{equation}
In particular, at this genus the ordinary gap holds if and only if the
relative gap holds. No genus $g$ gap is needed to determine the polar
part of $\cC_g$.
\end{proposition}
\begin{proof}
The known coefficients of $\widehat R$ through $\h^{2g-2}$ agree with
$\Phi+A^{<g}$, where $A^{<g}$ is regular in $t$. Set all higher
coefficients of this auxiliary input equal to those of $\Phi$.
\Cref{prop:geometric-transfer} shows that this extension does not change
$\cC_g$, since its coefficient involves only $R_0,\ldots,R_{g-1}$.
Strong regularity and \cref{lem:Bernoulli} give
\[
 \Pol\cC_g=\bigl(u_g-(-1)^gr_g\bigr)t^{2-2g}.
\]
Taking principal parts in $U_g=(-1)^gR_g+\cC_g$ and subtracting $u_gt^{2-2g}$
proves \eqref{eq:polar-defect-identity}.
\end{proof}

\begin{corollary}[Finite genus transfer]\label[corollary]{cor:finite-genus}
For any $G\ge2$, under the hypotheses of \cref{thm:transfer-intro},
the ordinary gaps for $2\le g\le G$ are equivalent to the relative gaps
for $2\le g\le G$. No boundary conditions above genus $G$ are required.
\end{corollary}
\begin{proof}
If the relative gaps are known through genus $G$, apply
\cref{prop:polar-defect} at each genus to obtain the ordinary gaps.
Conversely, start at $g=2$, where the lower genus assumption is empty.
The same proposition gives the relative genus two gap. Induction on
$g\le G$ proves the remaining relative gaps.
\end{proof}

\begin{proof}[Proof of \cref{thm:transfer-intro}]
Apply \cref{cor:finite-genus} for every finite $G$. A meromorphic germ
with zero principal part is holomorphic, so the conclusion has the
claimed fixed genus analytic meaning.
\end{proof}

\begin{proof}[Proof of \cref{thm:main}]
The low genus conditions and meromorphicity were established in
\eqref{eq:low} and the paragraph following it.
\Cref{prop:geometric-transfer} gives the exact correspondence required for
\cref{thm:transfer-intro}, and \eqref{eq:nome-linear} gives the simple zero.
It remains to normalize the ordinary gap input.

By \cref{lem:period-scale}, Brini's coordinate is exactly
$t_{\mathrm{CF}}=T=\sqrt3t$. His ordinary theorem, identified explicitly
before \eqref{eq:ordinary-input}, reads
\[
 U_g=\frac{3^{g-1}B_{2g}}{2g(2g-2)}T^{2-2g}+O(1).
\]
Since $3^{g-1}T^{2-2g}=t^{2-2g}$, this is exactly the ordinary gap with
coefficient $u_g$ in \eqref{eq:coefficients}. Applying
\cref{thm:transfer-intro} proves \eqref{eq:main-gap}.
\end{proof}

\begin{remark}
The induction does not use a genus-$g$ relative boundary condition.
Only lower genus relative germs enter $\cC_g$, and the Bernoulli
convolution proves the compatibility of their singular model with the
ordinary input. This distinction is essential for the reconstruction in
\cref{sec:jets}, where the resulting gap is subsequently used to determine
holomorphic ambiguities.
\end{remark}

\section{Finite parts, regular coefficients, and reconstruction}\label{sec:jets}

\subsection{A finite part relation}
The factor $t$ in \eqref{eq:strong} strengthens the information obtained
from a gap alone. Let $\bR_r,\bU_g$ be as in
\cref{sec:intro}, and put $A_r=(-1)^r\bR_r$. Define
\begin{equation}\label{eq:N-genus}
 \cN_g=[\h^{2g-2}](\cC-V),\qquad g\ge2.
\end{equation}
Then $\cN_g\in t\CC[[t]]$.

\begin{proposition}[Finite part identity]\label[proposition]{prop:finitepart}
For every $g\ge2$,
\begin{equation}\label{eq:finitepart}
 \bU_g(0)-(-1)^g\bR_g(0)
 =\sum_{j=1}^g b_j(-1)^{g-j}\bR_{g-j}^{(2j)}(0).
\end{equation}
More precisely, the difference between the corresponding functions of $t$
and the sum of regular derivatives is divisible by $t$.
\end{proposition}
\begin{proof}
At degree $\h^{2g-2}$ in \eqref{eq:V-formula}, one has $j+r=g$.
The universal part is exactly $\gamma_gt^{2-2g}$ by
\cref{lem:Bernoulli}. The remainder of the vacuum term is
\[
 \sum_{j=1}^g b_j A_{g-j}^{(2j)}(t)
 =\sum_{j=1}^g b_j(-1)^{g-j}\bR_{g-j}^{(2j)}(t).
\]
The term $-b_1R_0''$ has degree zero and does not occur for $g\ge2$.
Subtract the universal poles from \eqref{eq:geom-intro} and use
$\cN_g(0)=0$.
\end{proof}

For $g=2$ this becomes
\begin{equation}\label{eq:finite2}
 \bU_2(0)-\bR_2(0)
 =\frac1{24}\bR_1''(0)+\frac7{5760}\bR_0^{(4)}(0).
\end{equation}
The identity concerns the full transformed potentials. It does not assert
that each original positive degree graph or each original enumerative
summand vanishes separately at the conifold.

\subsection{Finite partition cutoffs for regular jets}
Let $\Dur(\lambda)$ be the number of diagonal boxes of $\lambda$.
The content polynomial has precisely one factor $t$ for each diagonal
box, and hence is divisible by $t^{\Dur(\lambda)}$.

\begin{proposition}[Regular jet cutoff]\label[proposition]{prop:jetcutoff}
To compute $[t^k\h^m](\cC-V)$ for $k\ge1$, it suffices to retain partitions
satisfying
\begin{equation}\label{eq:jetcutoff}
 |\lambda|\le k+m,\qquad \Dur(\lambda)\le k,
\end{equation}
form the truncated sum in \eqref{eq:weight-content}, and then take its
truncated logarithm and subtract $\log P(\rho)$.
\end{proposition}
\begin{proof}
The size bound is \cref{lem:cutoff}. A partition with Durfee rank greater
than $k$ has no term of $t$-degree at most $k$, because all remaining factors
in \eqref{eq:weight-content} have nonnegative $t$-degree. In the logarithm,
each nonconstant factor has positive $t$-degree, so omitted terms cannot
return through products to affect a smaller $t$-degree. All operations may
therefore be performed with the stated finite truncation.
\end{proof}

The order of operations in this proposition matters. For $k>1$, taking a
logarithm involves products of partition contributions; one must not
replace it by the sum of logarithms of individual weights.

\subsection{The hook formula for the first regular jet}
Only hooks have Durfee rank one. Write a hook as
$\lambda=(a+1,1^b)$, with $a,b\ge0$, and let $d=a+b+1$. Its content
polynomial satisfies
\begin{equation}\label{eq:hookcontent}
 [t]\prod_{\square\in\lambda}(t+\h c(\square))
 =(-1)^ba!b!\h^{a+b}.
\end{equation}
The row differences for a hook telescope further. At $t=0$,
\begin{equation}\label{eq:hookK}
 \widetilde K_{a,b}(0,\h)
 =\frac{A'(\h(a+1/2),\h)-A'(-\h(b+1/2),\h)}\h
   -dA_0''(0).
\end{equation}

\begin{proposition}[All order first jet expression]\label[proposition]{prop:hook}
Let $\rho_1=v(0)$. Then
\begin{equation}\label{eq:allhook}
 [t](\cC-V)=
 \sum_{a,b\ge0}\rho_1^{a+b+1}(-1)^ba!b!\h^{a+b}
              e^{\widetilde K_{a,b}(0,\h)}-\rho_1.
\end{equation}
This is a well defined even series with zero constant term. In particular,
\begin{equation}\label{eq:hook2}
 [t\h^2](\cC-V)
 =\rho_1\left(A_1''(0)+\frac{A_0^{(4)}(0)}{24}\right)
   +2\rho_1^2A_0'''(0)+3\rho_1^3.
\end{equation}
\end{proposition}
\begin{proof}
Since $Z_A\in1+tK[[t,\h]]$, its logarithm and $Z_A-1$ have the same
coefficient of $t$. Equations \eqref{eq:hookcontent} and \eqref{eq:hookK}
therefore give the sum in \eqref{eq:allhook}. Also
$[t]\log P(\rho)=\rho_1$. Finite coefficient validity follows because only
$a+b\le m$ can contribute to the coefficient of $\h^m$.

At degree $\h^2$, the size one hook contributes
$\rho_1(A_1''+A_0^{(4)}/24)$. The two size-two hooks have opposite leading
content coefficients; their first regular exponent terms together give
$2\rho_1^2A_0'''$. The three size-three hooks have constant exponent one
at $\h=0$, and their signed factorial coefficients sum to
$2-1+2=3$. Larger hooks do not contribute. This proves \eqref{eq:hook2}.
Evenness and the zero constant term also follow from \cref{thm:strong}.
\end{proof}

\subsection{Constructive determination of the holomorphic ambiguity}
Let $\mathscr R_g(X,S)$ denote the common differential ring representative;
its conifold evaluation is $R_g$. At large radius,
$X=1-27Q+O(Q^2)$ and $S=3Q+O(Q^2)$. Put
\begin{equation}\label{eq:theta-reconstruction}
 \Theta=X(X-1)\partial_X+
 \left(-S^2+\frac{X-1}{3}S-\frac{X(X-1)}9\right)\partial_S.
\end{equation}
Theorem~1.8 of \cite{BFGW}, with
$F_{g,1}=I_{11}^{-1}\theta F_g$, gives
\begin{equation}\label{eq:HAE-reconstruction}
 \partial_S\mathscr R_g=\frac3{2X}
          \sum_{j=1}^{g-1}(\Theta\mathscr R_j)(\Theta\mathscr R_{g-j}),
 \qquad \Theta\mathscr R_1=-\frac X{24}.
\end{equation}
The last expression specifies the genus one input without requiring a
polynomial representative for its logarithm. Let $\mathsf P_g$ be the
$S$-antiderivative of the right side with zero $S$-constant term. The
known finite generation and regularity constraints imply
\begin{equation}\label{eq:ambiguity-space}
 \mathscr R_g=\mathsf P_g+h_g(X),\qquad
 h_g\in K[X]_{\le m},\qquad m=2g-2.
\end{equation}
Indeed the ambiguity in the modular presentation is
$\mC^{-m}K[\mA,\mC]_{3m}$, with weights $1,3$, whose basis is precisely
$1,X,\ldots,X^m$. This is the ambiguity space already identified in
\cite[Section~1.6]{BFGW}.

\begin{proposition}[Explicit reconstruction after the gap]\label[proposition]{prop:reconstruction}
For every $g\ge2$, the lower genus potentials, the gap of
\cref{thm:main}, and $\mathscr R_g(1,0)=0$ determine $\mathscr R_g$
uniquely. More explicitly, define $\alpha_1,\ldots,\alpha_m$ by
\begin{equation}\label{eq:ambiguity-algorithm}
 \operatorname{PP}_{\xi=0}\!
 \left(r_gt(\xi)^{-m}-\mathsf P_g(\xi^{-1},S_c(\xi))\right)
     =\sum_{j=1}^m\alpha_j\xi^{-j},\qquad
 \mathsf H_g(X)=\sum_{j=1}^m\alpha_jX^j.
\end{equation}
Then
\begin{equation}\label{eq:reconstruction-result}
 \mathscr R_g=\mathsf P_g+\mathsf H_g
                       -\mathsf P_g(1,0)-\mathsf H_g(1).
\end{equation}
\end{proposition}
\begin{proof}
On the quotient $K[X]_{\le m}/K$, the principal part map in the $t$
coordinate is triangular in the bases $X,\ldots,X^m$ and
$t^{-1},\ldots,t^{-m}$. Since $\xi=\sqrt3t+O(t^2)$, its determinant is
\begin{equation}\label{eq:ambiguity-determinant}
 \prod_{j=1}^m3^{-j/2}=3^{-m(m+1)/4}\ne0.
\end{equation}
Here one may take $K=\CC$, or an appropriate characteristic zero field
containing $\sqrt3$. Thus the gap fixes the ambiguity modulo a constant.
The remaining constant is fixed by its value at $(1,0)$.
In the $\xi$ coordinate the same map has the basis $\xi^{-1},\ldots,
\xi^{-m}$, and is the identity. Taking its coefficients yields
\eqref{eq:ambiguity-algorithm}; adjusting the constant gives
\eqref{eq:reconstruction-result}. Existence follows from the actual
relative potential and its proven gap, not from assuming an arbitrary
formal integration satisfies the geometric constraints.
\end{proof}

The conditional uniqueness observation is already in \cite{BFGW}. The
role of the present result is to make its gap input available, not to
claim a new anomaly equation or a new general principle of direct
integration. In particular, this consequence fits into the established
boundary-data method of \cite{HKGauge,HKQ,GKMW,HKR}. At fixed genus, only finitely many coefficients of the
period and $S_c$ are required by \eqref{eq:ambiguity-algorithm}.

The accompanying exact computation reconstructs genera two through four.
For example, the $S$-independent term at genus four is
\begin{align}\label{eq:h4}
 h_4(X)&=-\frac{X-1}{12541132800}\Bigl(
     29725884X^5-80836858X^4+76970853X^3\notag\\
   &\hspace{40mm}{}-29502608X^2+3690240X-44928\Bigr).
\end{align}
The full $S$-polynomial is included in the exact companion data.

For a check not used in this reconstruction, the degree one sheaf space
is the dual projective plane: a pure stable sheaf of degree one and
Euler characteristic one is $\mathcal O_L$ for a line $L$. Its centered
cohomology polynomial is $y^{-1}+1+y$. The degree one case of the
sheaves/relative GW correspondence \cite{Bousseau,BFGW} therefore gives
\begin{equation}\label{eq:independent-degree-one}
 \sum_{g\ge0}(-1)^g[Q]\mathscr R_g\,\h^{2g-1}
       =\frac{1+2\cos\h}{2\sin(\h/2)}.
\end{equation}
There is no multiple cover ambiguity in degree one. The coefficients
computed from the reconstructed polynomial, using
$[Q]\mathscr R_g=(-27\partial_X+3\partial_S)\mathscr R_g|_{(1,0)}$,
are respectively
\begin{equation}\label{eq:degree-one-values}
 [Q]\mathscr R_2=\frac{29}{640},\qquad
 [Q]\mathscr R_3=\frac{137}{322560},\qquad
 [Q]\mathscr R_4=\frac{41}{7372800}.
\end{equation}
They agree exactly with \eqref{eq:independent-degree-one}.
The reconstructed genus four gap is \emph{not} an independent test of
the gap, because it was used in the construction. Equation
\eqref{eq:independent-degree-one}, which was not an input, is a separate
check on the resulting large radius data.

\subsection{Genus two from independent geometric expressions}
The genus two formulas used here are established expressions in the
large radius differential ring \cite[Theorem~1.5 and Section~4.6]{BFGW}:
\begin{align}\label{eq:R2source}
 R_2&=\frac{XS}{384}-\frac{X^2}{360}+\frac{X}{240}-\frac1{720},\\
 U_2&=X^{-1}\left(\frac58S^3+\frac X8S^2+\frac{X^2}{96}S
             +\frac{X^3}{4320}+\frac{X^2}{4320}-\frac X{2160}\right).
 \label{eq:U2source}
\end{align}
They are evaluated using \eqref{eq:evaluation}; no conifold gap is imposed
to determine their coefficients. From \eqref{eq:PF},
\begin{equation}\label{eq:Sseries}
 S^c=-\frac1{3\xi}+\frac19+\frac{2}{81}\xi
                 +\frac{28}{2187}\xi^2+O(\xi^3).
\end{equation}
For example, the relative expression gives
\[
 R_2=-\frac7{1920}\xi^{-2}+\frac{77}{17280}\xi^{-1}+O(1).
\]
If $f=a\xi^{-2}+b\xi^{-1}+O(1)$, \eqref{eq:T-expansion} gives
\[
 f=\frac a3t^{-2}+\frac{11a/9+b}{\sqrt3}t^{-1}+O(1).
\]
For the displayed $a,b$, the second coefficient is zero, recovering
$R_2=-7/(5760t^2)+O(1)$.

There is a stronger check on the regular terms. Direct expansion of
\eqref{eq:R0-R1}, with the universal singular pieces subtracted, yields
\begin{equation}\label{eq:derivative-values}
 \bR_0'''(0)=-\frac{\sqrt3}{18},\qquad
 \bR_0^{(4)}(0)=\frac1{324},\qquad
 \bR_1''(0)=\frac1{7776}.
\end{equation}
Thus \eqref{eq:finite2} predicts
\begin{equation}\label{eq:finite-value}
 \FP(U_2-R_2)=\frac1{24\cdot7776}
                   +\frac7{5760\cdot324}
             =\frac{17}{1866240}.
\end{equation}
Expansion of \eqref{eq:R2source}--\eqref{eq:U2source}, independently of the
vacuum formula, gives exactly
\begin{equation}\label{eq:C2-explicit}
 \cC_2=-\frac{17}{5760t^2}+\frac{17}{1866240}
        -\frac{521\sqrt3}{12597120}t+O(t^2).
\end{equation}
The regular vacuum part has the same linear coefficient. Equivalently,
using $A_0=\bR_0$, $A_1=-\bR_1$, and $\rho_1=\sqrt3/27$ in
\eqref{eq:hook2} gives
\begin{align*}
 [t]\cN_2
 &=\frac{\sqrt3}{27}
       \left(-\frac1{7776}+\frac1{24\cdot324}\right)
   +2\left(\frac{\sqrt3}{27}\right)^2
       \left(-\frac{\sqrt3}{18}\right)
   +3\left(\frac{\sqrt3}{27}\right)^3=0.
\end{align*}
In this specific example $\cN_2=O(t^2)$, stronger than the general $O(t)$
statement. It is worth mentioning that we do not infer a general all genus strengthening from this
low genus cancellation.

\subsection{Reproducibility}
Supplementary material~1 (\texttt{ESM\_1.zip}) contains exact arithmetic programs for the
elliptic cumulants, the content product identity, the cusp normalization,
the finite part relation and the reconstruction through genus four.
For example, with
\[
 \rho=2t(1+t),\qquad
 A=\frac{t^3}{3}+\frac{t^4}{5}
   +\h^2\left(\frac{t^2}{7}+\frac{t^3}{11}\right)
   +\frac{\h^4t^2}{13},
\]
the direct shifted-power-sum expansion and the content product expansion
agree for all $t^k\h^m$ with $k\le3$, $m\le6$, using the $97$ partitions
of size at most nine. Both give
$[t\h^2](\cC-V)=1434/35$. The size cutoff is justified by
\cref{lem:cutoff}, so each comparison is finite.

The verification also distinguishes normalization errors which are
invisible to the polar part alone. For $R_0=f_0$, $R_1=(1/24)\log t$ and
$\rho=\nu t$, omitting the factor $1/2!$ in the two-insertion cumulant
changes the finite part of $\cC_2$ by $\nu^2$ without changing its polar
part. Thus the finite part identity provides a separate test of the
geometric normalization. The proofs of the regularity and transfer
theorems do not depend on these computations.

\medskip
\noindent\textbf{Supplementary material 1.} Exact arithmetic verification scripts,
input formulas, reconstructed genus two to genus four polynomials, and
recorded test outputs for the completed cycle transfer.

\section*{Statements and declarations}
\noindent\textbf{Data and code availability.}
The exact arithmetic verification programs, their input formulas,
reconstructed polynomial data and recorded outputs are supplied in the
accompanying Supplementary material~1 (\texttt{ESM\_1.zip}). No empirical data sets are used.
The mathematical proofs do not depend on the execution of these programs.

\medskip\noindent\textbf{Funding.} The author is supported by the National Natural Science Foundation of China
(NSFC) under Grant Nos.~11501470, 11426187, and 11791240561, and is partially
supported by NSFC Grant No.~11671328, the Chengdu Science and Technology
Program under Grant No.~2025-YF09-00007-SN, and the Fundamental Research
Funds for the Central Universities under Grant Nos.~2682021ZTPY043,
2682025ZTPY001, 2682025ZTPY057, and 2682025ZTO002.

\medskip\noindent\textbf{Competing interests.} The author has no relevant financial or non-financial interests to disclose.

\section*{Acknowledgement.} 
The author would like to thank Sung-Soo Kim, Kimyeong Lee, Futoshi Yagi,
Satoshi Nawata, and Rui-Dong Zhu for useful discussions and collaborations in
mathematical physics. The author is particularly grateful to Andrea Brini, Todar Milanov 
and Arpan Saha for inspiring online seminars during the pandemic, and to Doan Nhat Minh and Van
Nguyen for enjoyable and fruitful collaborations. The author also thanks
Bohui Chen, An-Min Li, and Guosong Zhao for their constant supports, as well
as friends and colleagues for stimulating discussions at various conferences.
Finally, the author gratefully acknowledges the hospitality and support of the
Mainz Institute for Theoretical Physics (MITP) of the Cluster of Excellence
PRISMA+ (Project ID 390831469).

\end{document}